\documentclass{article}%
\usepackage{fullpage}
\usepackage{amsmath}
\usepackage{amssymb}
\usepackage{amsfonts}
\usepackage{amsthm}
\usepackage{mathtools}
\usepackage{graphics}
\usepackage{graphicx}
\usepackage{pstricks,multido}
\usepackage{pstricks-add}
\usepackage{pst-node}
\usepackage{pst-plot}
\usepackage{xcolor}
\usepackage{float}
\usepackage{varioref}
\usepackage{array,makecell}%
\providecommand{\U}[1]{\protect\rule{.1in}{.1in}}
\allowdisplaybreaks

\makeatletter
\def\@listi{\leftmargin\leftmargini
\parsep 4.5\p@ \@plus2\p@ \@minus\p@
\topsep 9\p@   \@plus3\p@ \@minus5\p@
\itemsep 0.0cm \@plus2\p@ \@minus\p@}
\let\@listI\@listi
\@listi
\newtheorem{theorem}{Theorem}

\newtheorem{notation}{Notation}
\newtheorem{definition}[notation]{Definition}
\newtheorem{lemma}[notation]{Lemma}

\newtheorem{corollary}[notation]{Corollary}

\newtheorem{example}[notation]{Example}
\newtheorem{remark}[notation]{Remark}

\usepackage[normalem]{ulem}

\begin{document}

\title{Maximum gap in cyclotomic polynomials}
\author{
Ala'a Al-Kateeb\footnote{Department of Mathematics, Yarmouk university, Irbid, Jordan\;\;\; (alaa.kateeb@yu.edu.jo)}\;\;\;\; Mary Ambrosino\footnote{Department of Mathematics, North Carolina State University, Raleigh, USA\;\;\;(mary.e.ambrosino@gmail.com)}\;\;\;\; Hoon Hong \footnote{Department of Mathematics, North Carolina State University, Raleigh, USA\;\;\; (hong@ncsu.edu)}\;\;\;\; Eunjeong Lee\footnote{Institute of Mathematical Sciences, Ewha Womans University, Seoul, Republic of Korea\;\;\;(ejlee127@gmail.com)}}
\date{\today}
\maketitle

\begin{abstract}
Cyclotomic polynomials play fundamental roles in number theory, combinatorics,
algebra and their applications. Hence their properties have been extensively investigated. 
In this paper, we study the maximum gap $g$ (maximum of
the differences between any two consecutive exponents). In 2012, it was shown
that $g\left(  \Phi_{p_{1}p_{2}}\right)  =p_{1} -1$ for primes $p_{2}>p_{1}$.
In 2017, based on numerous calculations, the following generalization was
conjectured: $g\left(  \Phi_{mp}\right)  =\varphi(m)$ for square free odd $m$
and prime $p>m$. The main contribution of this paper is a proof of this conjecture.

\end{abstract}



\section{Introduction}

The $n$-th cyclotomic polynomial $\Phi_{n}$ is defined as the monic polynomial
in $\mathbb{Z}[x]$ whose complex roots are the primitive $n$-th~roots of
unity. The cyclotomic polynomials play fundamental roles in number theory,
algebra, combinatorics and their applications, which motivated the extensive
investigation on its structure, for instance `height', `jump', and `gap'.

The investigation on height (maximum absolute value of coefficients) was
initiated by the finding that the height can be bigger than 1 (as exhibited by
$n=105$). 
It has produced numerous results, to list a few: 
upper bound~\cite{BL,BE3,BG1,BZ1,VA,VA3,VA2,BAT2,VA4,BZ2,LE},
realizability~\cite{VA5,KO,Fi,FO,GA-Mo,GA-MO2,Mor2003,Mor2012,Ji}, and
flatness~\cite{BAC,ED,KA1,KA2,ZH17}. 
The investigation on jumps (a jump happens when two
consecutive coefficients  are different) was initiated 
by Bzd\c{e}ga \cite{BZ4}
which was motivated by the work of Gallot and Moree \cite{GA-MO2}. A sharp bound of the
number of jumps for ternary cyclotomic polynomials is provided in
\cite{CCLMS2016}. 

The investigation on maximum gap (maximum difference between the consecutive exponents, see Definition~\ref{def:maxgap}) was
motivated by its need for analyzing the complexity~\cite{HLLP11a} of a certain
cryptographic pairing operation over elliptic curves~\cite{FST,LLP,ZZH}. Later
it became a problem on its own because it could be viewed as a first step
toward understanding of sparsity structure of cyclotomic polynomials. In 2012
\cite{HLLP}, it was shown that the maximum gap for binary cyclotomic
polynomial $\Phi_{p_{1}p_{2}}$ is $p_{1}-1$, that is, $g(\Phi_{p_{1}p_{2}})=
p_{1}-1$. In 2014, Moree~\cite{Moree2014} revisited the result and provided an
inspiring conceptual proof by making a connection to numerical semigroups of embedding dimension two. In 2016, Zhang~\cite{Zhang16} gave a simpler proof, along with
the result on the number of occurrences of the maximum gaps. Concurrently,
Camburu, et al. ~\cite{CCLMS2016} communicated an even simpler proof by Kaplan.
Thus we feel that the investigation of maximum gap of binary cyclotomic
polynomials is practically completed. Hence, the next natural challenge is to
study the maximum gap of ternary cyclotomic polynomials, and ultimately,
arbitrary cyclotomic polynomials.

In this paper, we tackle the maximum gap in arbitrary case.
The following
three graphs visualize the relation between $n$ and $g(\Phi_{n})$, where the
horizontal axis stands for~$n$ and the vertical axis for $g(\Phi_{n}).$

\begin{center}
{\small {
\begin{tabular}
[c]{ccc}%
(A) arbitrary $n$ & (B) odd square-free $n$ & (C) $n=mp$ for $m=105, ~p > 7 $
prime\\
\includegraphics[width=2.0in]{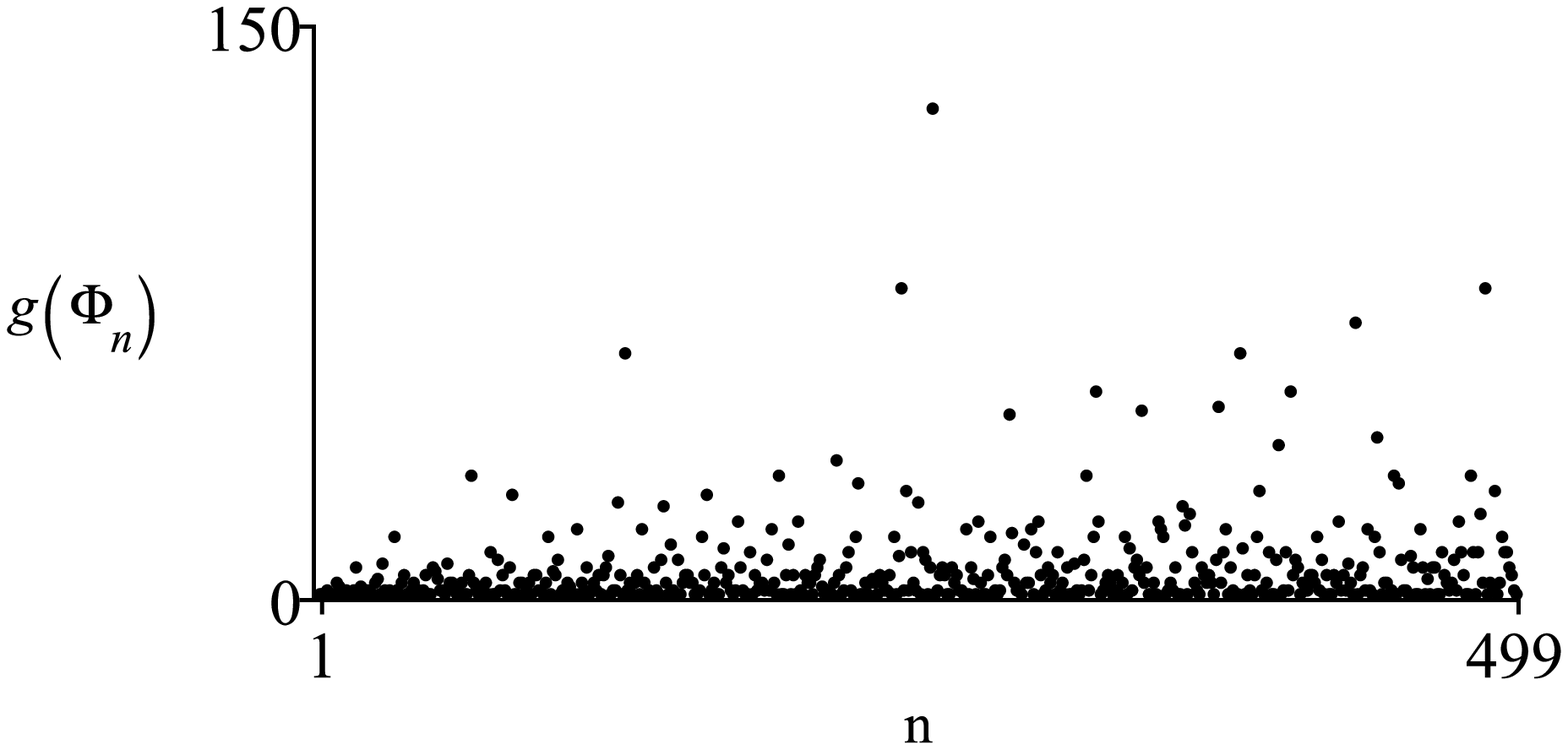} &
\includegraphics[width=2.0in]{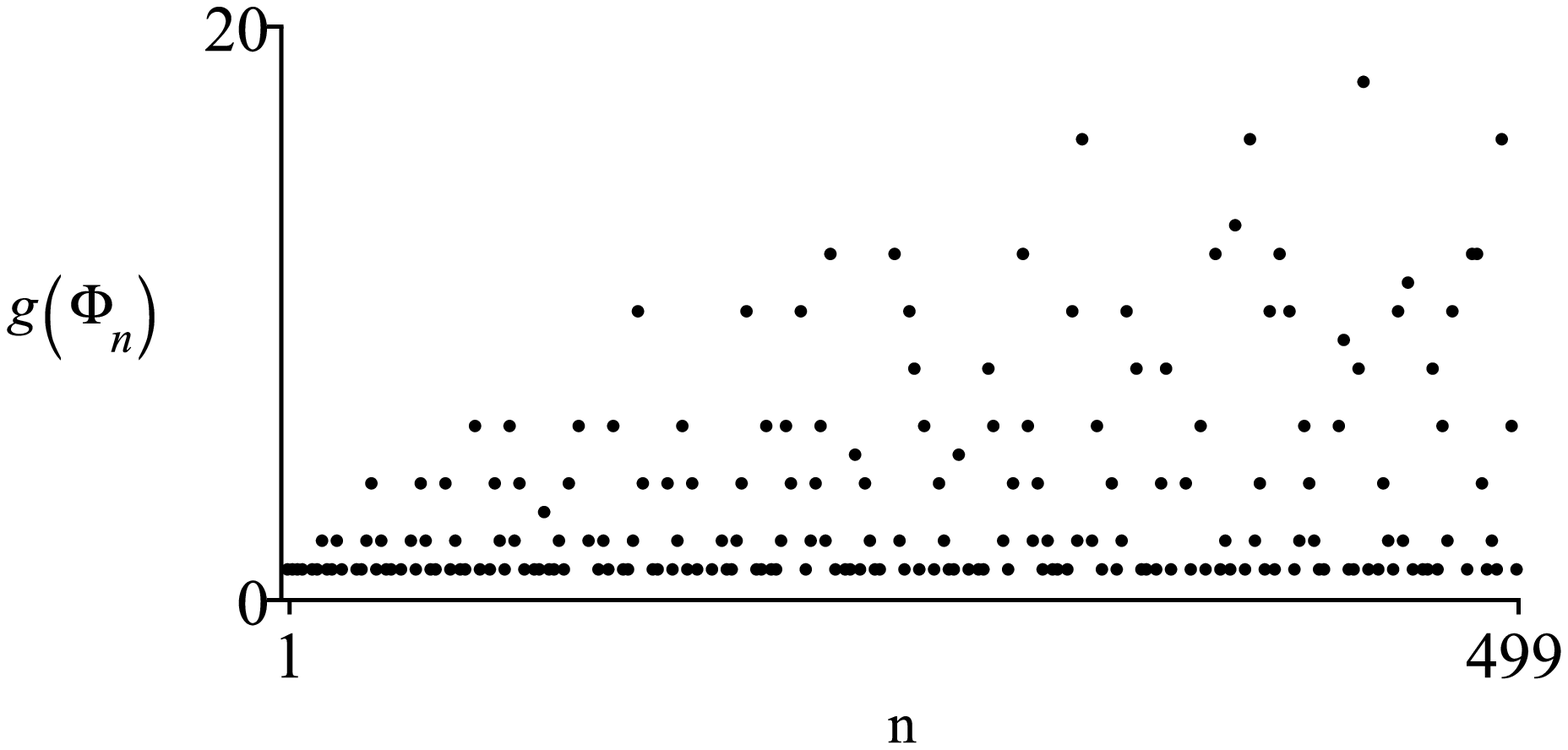} &
\includegraphics[width=2.0in]{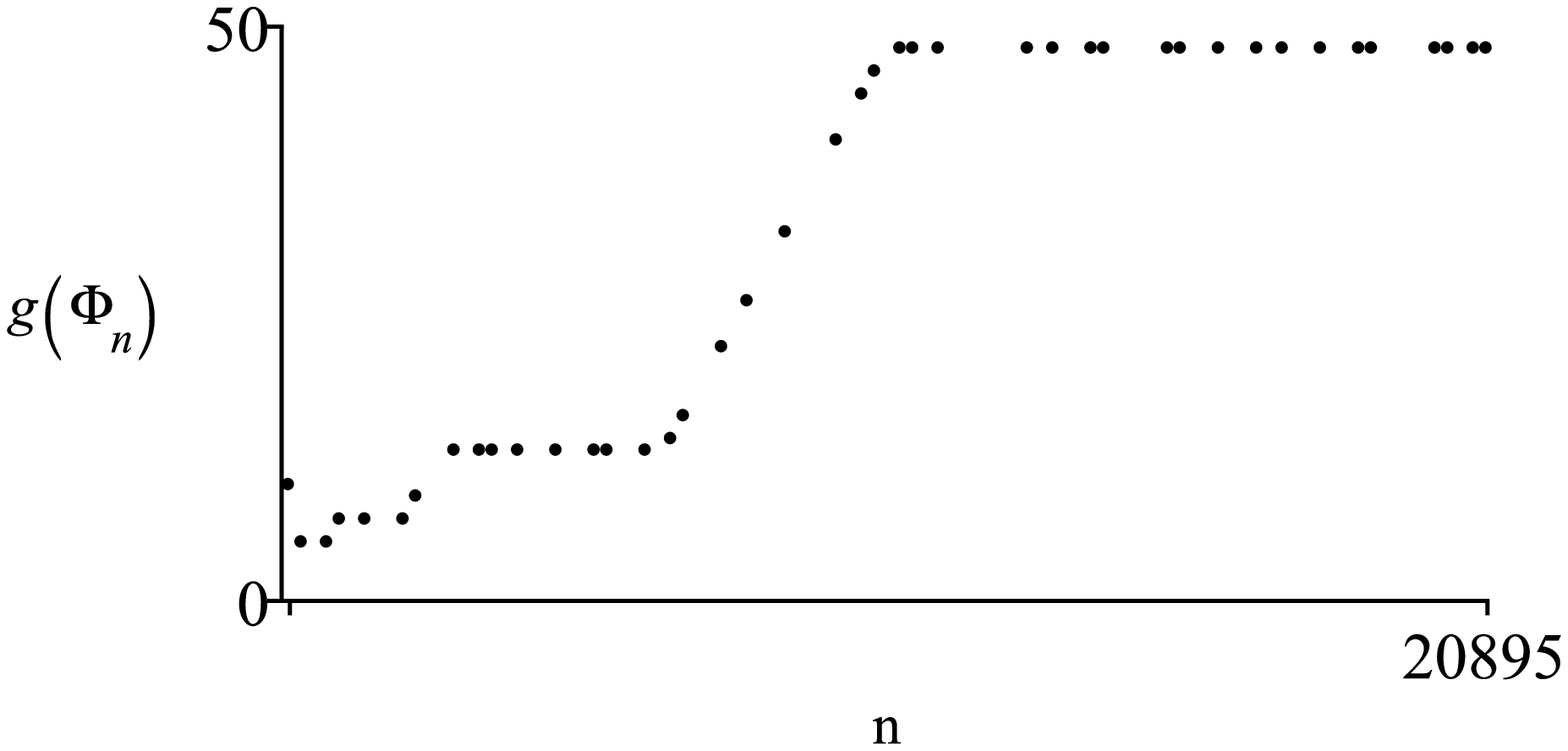}
\end{tabular}
} }
\end{center}

Graph A shows $g(\Phi_{n})$ for arbitrary $n$ up to $499$. It is not easy to
discern any pattern on the relation between $n$ and $g(\Phi_{n})$, mainly
because there are several very large values of $g(\Phi_{n})$, pushing down
other values toward~$0$. Close inspection shows that those $n$ are usually
non-square-free. In fact it can be easily explained by the basic property of
cyclotomic polynomial
\[
\Phi_{n}\left(  x\right)  =\Phi_{\mathrm{radical}(n)}\left(  x^{\frac
{n}{\mathrm{radical}(n)}}\right)  \ \ \ \ \ \ \text{which implies  \ \ \ \ }%
g(\Phi_{n})=\frac{n}{\mathrm{radical}(n)}\ g(\Phi_{\mathrm{radical}(n)})
\]
The factor $\frac{n}{\mathrm{radical}(n)}$ contributes to the largeness. Thus
we will restrict the study to square-free $n$. But then we recall another
basic property
\[
\Phi_{2n}\left(  x\right)  =\pm\Phi_{n}\left(  -x\right)  \ \ \ \ \ \ \text{in
turn \ \ \ \ \ }g(\Phi_{2n})=g(\Phi_{n}).
\]
Hence we will further restrict the study to odd square-free $n$.

Graph B shows $g(\Phi_{n})$ for only odd square-free $n$ up to $499$. Now we
see some patterns: there are several horizontal lines, that is, several values
of $n$ have same values for $g(\Phi_{n})$. Close inspection shows that many of
them have same prime factors, except the biggest prime factor. This motivates
us to restrict $n$ to be $mp$ for a fixed odd square-free $m$ and arbitrary
prime $p$ which is bigger than all prime factors of $m$.

Graph C shows $g(\Phi_{n})$ for $n$ where $n=mp$ for $m=105$ and arbitrary
prime $p$ which is bigger than all prime factors of $m$. Now we finally see
some striking patterns. The graph appears to be a piece-wise linear curve with
several flat line segments.  In particular, there is a flat line apparently
forever to the right. Close observation shows that the flat line starts from
the smallest prime $p>m$ and the value is $\varphi(m)$. In
2017~\cite{APhD,AHL2017-lower}, through numerous computational confirmation,
it was conjectured that the phenomena always hold. In fact, it is consistent
with the previous result in \cite{HLLP},
 that is, for any two primes $p_1<p_2$, we have  $g(\Phi_{p_{1}p_{2}}%
)=p_{1}-1=\varphi(m)$, when we let $m=p_{1}$ and $p=p_{2}$. 
The main contribution of this paper is to prove this conjecture.

\begin{theorem}
[Main]\label{main} If $m$ is a square-free odd integer and $p>m$ is  a prime, then \[ g(\Phi_{mp})=\varphi(m)\]
\end{theorem}

For proving this result, we began by reviewing various proof techniques used
for studying cyclotomic polynomials. We noticed that a fruitful technique has
been dividing the $\Phi_{mp}$ into several blocks (Notation \ref{blocks}) and
investigate and exploit their \emph{relations}. Several nice relations (such
as repetition and invariance) were observed by~\cite{AM2,KA1,KA2}.
See~\cite{AK, AHL17} for a list of those relations. Those relations have been
successively used for computing cyclotomic polynomials efficiently by Arnold
and Monagan \cite{AM2} and for studying flatness by Kaplan~\cite{KA1,KA2}.

Encouraged by the effectiveness of the relations, we also applied them to the
study of maximum gap. The relations allowed us to focus on a few
representative blocks instead of all the blocks of the cyclotomic polynomial.
However, the relations were not enough for proving the conjecture, mainly due
to the lack of understanding of the representative blocks. In~\cite{AK,
AHL17}, an explicit expression for the blocks was derived. The explicit
expression for the blocks allowed us to prove the conjecture for a certain
family of cyclotomic polynomials. However, the explicit expression was not
enough for proving the full conjecture, mainly because it is a sum of several
polynomials. The summation could cause introduction and cancellation of terms,
resulting in shrinking, enlarging, splitting of gaps, which is very difficult
to analyze.

Then, we hit on the \emph{key} structural observation: \emph{All
representative blocks are divisible by the $m$-th inverse cyclotomic
polynomial} (for the definition of inverse cyclotomic polynomial, see~\cite{Mor09} and Definition~\ref{def:icyc}). 
We captured the observation by deriving another expression for
blocks, that makes the divisibility explicit (Theorem~\ref{explicit}). Using
the new explicit expression for the blocks along with the known relations
among blocks, we were able to show that the maximum gaps within and between
blocks are all less than or equal to $\varphi(m)$, finally proving the
conjecture. We hope that the structural result (Theorem~\ref{explicit}) could
be useful for investigating other properties of cyclotomic polynomials.

The paper is structured as follows: In section \ref{sec:blocks} we review a
usual division of $\Phi_{mp}$ into several blocks (Notation~\ref{blocks}) and
list several known or easily provable relations among blocks
(Lemma~\ref{relation}). Then we prove the key structural result on blocks
(Theorem~\ref{explicit}). In section \ref{sec:gaps}, using the key structural
result for blocks along with the relations among blocks, we show that the
maximum gaps within and between blocks are all less than or equal to
$\varphi(m)$, finally proving the conjecture.

\section{Blocks}

\label{sec:blocks} In this section we review a usual division of $\Phi_{mp}$
into several blocks. Then we study each block and their relations.
Lemma~\ref{relation} list some known or new relations among blocks, explicitly
and concisely. Theorem~\ref{explicit} contains the key finding on the
structure of each block. It provides an explicit expression, which shows that
almost all blocks are divisible by an inverse cyclotomic polynomial.

\begin{notation}
[Division and blocks]\label{blocks}We divide $\Phi_{mp}$ as follows:%
\begin{align*}
\Phi_{mp}  &  =\sum_{i=0}^{\varphi\left(  m\right)  -1}f_{m,p,i}\ x^{ip} &  &
\text{where }\deg f_{m,p,i}<p\\
f_{m,p,i}  &  =\ \ \sum_{j=0}^{q}\ \ f_{m,p,i,j}\ x^{jm} &  &  \text{where
}\deg f_{m,p,i,j}<m
\end{align*}
where $q=\operatorname*{quo}\left(  p,m\right)  $ and $r=\operatorname*{rem}%
\left(  p,m\right)  $. Observing the sizes of each blocks, we call

\begin{enumerate}
\item $f_{m,p,i}$ a $p$-\emph{block}

\item $f_{m,p,i,j}$ an $m$-\emph{block} for $j<q$

\item $f_{m,p,i,q}$ an $r$-\emph{block}
\end{enumerate}
\end{notation}

\begin{example}
\label{example_block} Let $m=15$ and $p=53$. Then $\varphi\left(  m\right)
=8,\ q=3$ and $r=8$. The following diagram illustrates the division of
$\Phi_{mp}$ into $p$-blocks and then the division of each $p$-block into three
$m$-blocks and one $r$-block.
\[
\noindent \scalebox{0.75}{
\begin{pspicture}(2,0.75)(19,4.8)

\definecolor{MyColor}{rgb}{0.95,0.95,0.95} 

\psframe[linestyle=none,dimen=outer,fillstyle=solid,fillcolor=MyColor](0.00,3.00)(21.80,3.50)
\pcline[linewidth=0.015,linestyle=dashed]{<->}(0.00,4.50)(10.60,4.50)\ncput*{$p$}
\psline[linewidth=0.015](0.00,0.75)(0.00,4.80)
\pcline[linewidth=0.015,linestyle=dashed]{<->}(0.00,4.00)(3.00,4.00)\ncput*{$m$}
\pcline[linewidth=0.015,linestyle=dashed]{<->}(3.00,4.00)(6.00,4.00)\ncput*{$m$}
\psline[linewidth=0.015](3.00,1.75)(3.00,4.30)
\pcline[linewidth=0.015,linestyle=dashed]{<->}(6.00,4.00)(9.00,4.00)\ncput*{$m$}
\psline[linewidth=0.015](6.00,1.75)(6.00,4.30)
\pcline[linewidth=0.015,linestyle=dashed]{<->}(9.00,4.00)(10.60,4.00)\ncput*{$r$}
\psline[linewidth=0.015](9.00,1.75)(9.00,4.30)
\pcline[linewidth=0.015,linestyle=dashed]{<->}(10.60,4.50)(21.20,4.50)\ncput*{$p$}
\psline[linewidth=0.015](10.60,0.75)(10.60,4.80)
\pcline[linewidth=0.015,linestyle=dashed]{<->}(10.60,4.00)(13.60,4.00)\ncput*{$m$}
\pcline[linewidth=0.015,linestyle=dashed]{<->}(13.60,4.00)(16.60,4.00)\ncput*{$m$}
\psline[linewidth=0.015](13.60,1.75)(13.60,4.30)
\pcline[linewidth=0.015,linestyle=dashed]{<->}(16.60,4.00)(19.60,4.00)\ncput*{$m$}
\psline[linewidth=0.015](16.60,1.75)(16.60,4.30)
\pcline[linewidth=0.015,linestyle=dashed]{<->}(19.60,4.00)(21.20,4.00)\ncput*{$r$}
\psline[linewidth=0.015](19.60,1.75)(19.60,4.30)
\psline[linewidth=0.015](21.20,0.75)(21.20,4.80)
\pcline[linewidth=0.015,linestyle=dashed]{<-}(21.20,4.50)(21.80,4.50)
\pcline[linewidth=0.015,linestyle=dashed]{<-}(21.20,4.00)(21.80,4.00)

\rput( 1.5,2.5){$f_{m,p,0,0}$}
\rput( 4.5,2.5){$f_{m,p,0,1}$}
\rput( 7.5,2.5){$f_{m,p,0,2}$}
\rput( 9.8,2.5){$f_{m,p,0,3}$}

\rput( 1.5,2.0){$m$-block}
\rput( 4.5,2.0){$m$-block}
\rput( 7.5,2.0){$m$-block}
\rput( 9.8,2.0){$r$-block}

\rput(11.75,2.5){$f_{m,p,1,0}$}
\rput(14.75,2.5){$f_{m,p,1,1}$}
\rput(17.75,2.5){$f_{m,p,1,2}$}
\rput(20.25,2.5){$f_{m,p,1,3}$}

\rput(11.75,2.0){$m$-block}
\rput(14.75,2.0){$m$-block}
\rput(17.75,2.0){$m$-block}
\rput(20.25,2.0){$r$-block}

\rput( 6,1.25){$f_{m,p,0}$}
\rput(16.6,1.25){$f_{m,p,1}$}

\rput( 6.0,0.75){$p$-block}
\rput(16.6,0.75){$p$-block}
\end{pspicture}
}

\]
\noindent where the long shaded strip stand for the terms in $\Phi_{mp}$
(increasing order in the exponents). In order to avoid going off the right
margin, we cut the diagram off at the margin.
\end{example}

\begin{remark}
In Notation \ref{blocks}, the upper bound $\varphi\left(  m\right)  -1$ for
the index $i$ is from the following observation:
\[
\operatorname*{quo}\left(  \deg\Phi_{mp},p\right)  =\operatorname*{quo}\left(
\varphi\left(  mp\right)  ,p\right)  =\operatorname*{quo}\left(
\varphi\left(  m\right)  \left(  p-1\right)  ,p\right)  =\varphi\left(
m\right)  -1
\]

\end{remark}


\noindent Now we \textquotedblleft conquer\textquotedblright\ the blocks. We
begin by finding relation between blocks (Lemma \ref{relation}). Then, using
the relation, we derive an explicit expression for blocks (Theorem
\ref{explicit}). This is the \textquotedblleft key\textquotedblright\ result,
since it will play a crucial role in the remainder of the paper. We end this
subsection, by deriving a few immediate consequences (Corollaries
\ref{degdiff} and \ref{fmpi0<>0}). To present all these results compactly, we
will introduce a few notations and notions.

\begin{notation}
[Block operations]\label{operation}Let $h\ =\ \sum\limits_{k=0}^{m-1}%
h_{k}x^{k}\ \cong\ \left(  h_{0},\ldots,h_{m-1}\right)  $. Let $0\leq s<m$.
Then \vskip-0.5em $%
\begin{array}
[c]{lllllll}%
\mathsf{truncate} & : & \mathcal{T}_{s}h & = & \operatorname*{rem}\left(
h,x^{s}\right)  & \cong & \left(  h_{0},\ldots,h_{s-1}\right) \\
\mathsf{rotate} & : & \mathcal{R}_{m,s}h & = & \operatorname*{rem}\left(
x^{m-s}h,x^{m}-1\right)  & \cong & \left(  h_{s},h_{s+1},\ldots,h_{m-1}%
,h_{0},\ldots,h_{s-1}\right)
\end{array}
$
\end{notation}

\begin{definition}
[Inverse cyclotomic polynomial]
\label{def:icyc}
The $m$-th inverse cyclotomic polynomial
$\Psi_{m}$ is defined as
\[
\Psi_{m}=\frac{x^{m}-1}{\Phi_{m}}%
\]
We will use the notation $\psi(m)=\deg\Psi_{m}$. See \cite{Mor09} for various
interesting properties.
\end{definition}

\noindent From now on, through out the paper, we will assume that $p>m$ and
$0\leq i\leq\varphi\left(  m\right)  -1$.

\begin{lemma}
[Relation between blocks]\label{relation}We have

\begin{enumerate}
\item $f_{m,p,i,0}=f_{m,p,i,j}\ \ $ for $0\leq j\leq q-1$

\item \thinspace$f_{m,p,i,q}=\mathcal{T}_{r}f_{m,p,i,0}$

\item $f_{m,p,i+1,0}=\mathcal{R}_{m,r}f_{m,p,i,0}-a_{i+1}\Psi_{m}$
\end{enumerate}
\end{lemma}

\begin{remark}
The above lemma shows that the blocks $f_{m,p,i,0}$ play role of ``representatives".
\end{remark}

\begin{figure}[t]
\begin{tabular} [c]{|l|}

\hline
\scalebox{0.8}{
\begin{pspicture}(0,-2.5)(19.75,2.00) 
\definecolor{MyColor}{rgb}{0.95,0.95,0.95} 
\def\polytwo{\psframe[linewidth=0.04,dimen=middle,fillstyle=solid,fillcolor=MyColor](0.0,0.3)(4.6,-0.3)
\psline[linewidth=0.04,linestyle=solid,linecolor=black](2.3,0.3)(2.3,-0.3)
} 
\pcline[linewidth=0.015,linestyle=dashed]{<->}(0,-1.6)(12.5,-1.6)\ncput*{$p$} 
\pcline[linewidth=0.015,linestyle=dashed]{<->}(0,-0.6)( 2.3,-0.6)\ncput*{$m$} 
\rput( 0.00, 0.00){\polytwo} 
\rput( 12.50,-1.00){\polytwo} 
\multirput( 1.30, 0.00)(2,0){2}{$C_0$} 
\multirput(13.70,-1.00)(2,0){2}{$C_1$} 
\rput(10,1.0){$\Phi_{m}$} 
\multirput( 4.9,0.00)(0.3,0.0){50}{$\boldsymbol{\cdot}$} 
\multirput(17.4,-1.00)(0.3,0.0){8}{$\boldsymbol{\cdot}$} 
\multirput(13.0,-1.75)(0.6,-0.075){5}{$\boldsymbol{\cdot}$} 
\end{pspicture}
}
\\\hline

\scalebox{0.8}{
\begin{pspicture}(-2.3,-7.00)(17,2.0)

\definecolor{MyColor}{rgb}{0.95,0.95,0.95} 
\def\polytwo{\psframe[linewidth=0.04,dimen=middle,fillstyle=solid,fillcolor=MyColor](0.0,0.3)(4.6,-0.3)
\psline[linewidth=0.04,linestyle=solid,linecolor=black](2.3,0.3)(2.3,-0.3)
} 

\def\polythree{\psframe[linewidth=0.04,dimen=middle,fillstyle=solid,fillcolor=MyColor](0.0,0.3)(6.9,-0.3)
\psline[linewidth=0.04,linestyle=solid,linecolor=black](2.3,0.3)(2.3,-0.3)
\psline[linewidth=0.04,linestyle=solid,linecolor=black](4.6,0.3)(4.6,-0.3)
}

\rput(1.25,1.25){$f_{m,p,i,0}$} 
\rput(-1.75, 0.00){\polytwo} 
\rput(0.25,-3.00){\polytwo} 
\multirput(-0.5, 0.00)(2,0){2}{$C_0$} 
\multirput(1.50,-3.00)(2.0,0){2}{$C_{i}$} 
\psline[linewidth=0.015,linestyle=dashed,linecolor=black](0.25,-3.75)(0.25, 0.75) 
\psline[linewidth=0.015,linestyle=dashed,linecolor=black](2.55,-3.75)(2.55, 0.75) 
\psline[linewidth=0.015,linestyle=dashed,linecolor=black](0.25, 0.75)(2.55, 0.75) 
\psline[linewidth=0.015,linestyle=dashed,linecolor=black](0.25,-3.75)(2.55,-3.75) 
\pcline[linewidth=0.015,linestyle=solid]{<-}(0.25,-3.85)(0.25, -6.25)
\rput(0.25,-6.50){$ip$} 
\pcline[linewidth=0.015,linestyle=dashed]{<->}(0.25,-5.75)(2.55,-5.75)\ncput*{$m$} 
\multirput(-2.50,  0.00)(0.3,0.0){3}{$\boldsymbol{\cdot}$} 
\multirput( 3.15,  0.00)(0.3,0.0){2}{$\boldsymbol{\cdot}$}
\multirput( 5.15, -3.00)(0.3,0.0){2}{$\boldsymbol{\cdot}$} 

\rput(6.8,1.25){$f_{m,p,i,j}$} 
\rput(3.75, 0.00){\polytwo} 
\rput(5.75,-3.00){\polytwo} 
\multirput(5.00, 0.00)(2,0){2}{$C_0$} 
\multirput(7.00,-3.00)(2,0){2}{$C_{i}$} 
\psline[linewidth=0.015,linestyle=dashed,linecolor=black]( 5.75,-3.75)(5.75, 0.75) 
\psline[linewidth=0.015,linestyle=dashed,linecolor=black]( 8.05,-3.75)(8.05, 0.75) 
\psline[linewidth=0.015,linestyle=dashed,linecolor=black]( 5.75, 0.75)(8.05, 0.75) 
\psline[linewidth=0.015,linestyle=dashed,linecolor=black]( 5.75,-3.75)(8.05,-3.75) 
\pcline[linewidth=0.015,linestyle=solid]{<-}(5.75,-3.85)(5.75, -6.25)
\rput(5.75,-6.50){$ip+jm$} 
\pcline[linewidth=0.015cm,linestyle=dashed]{<->}(5.75,-5.75)(8.05,-5.75)\ncput*{$m$} 
\multirput(8.60, 0.00)(0.3,0.0){2}{$\boldsymbol{\cdot}$} 
\multirput(10.55,-3.00)(0.3,0.0){2}{$\boldsymbol{\cdot}$}

\rput(11.75,1.25){$f_{m,p,i,q}$}
\rput(9.25, 0.00){\polythree} 
\rput(11.25,-3.00){\polytwo} 
\rput(12.25,-4.50){\polytwo} 
\multirput(10.9, 0.00)(2,0){3}{$C_0$} 
\multirput(12.80,-3.00)(2,0){2}{$C_{i}$} 
\multirput(13.75,-4.50)(2,0){2}{$C_{i+1}$} 
\psline[linewidth=0.015,linestyle=dashed,linecolor=black](11.25,-3.75)(11.25, 0.75) 
\psline[linewidth=0.015,linestyle=dashed,linecolor=black](12.25,-3.75)(12.25, 0.75) 
\psline[linewidth=0.015,linestyle=dashed,linecolor=black](11.25, 0.75)(12.25, 0.75) 
\psline[linewidth=0.015,linestyle=dashed,linecolor=black](11.25,-3.75)(12.25,-3.75) 
\pcline[linewidth=0.015,linestyle=solid]{<-}(11.25,-3.85)(11.25, -6.25)
\rput(11.05,-6.50){$ip+qm$} 
\pcline[linewidth=0.015cm,linestyle=dashed]{<->}(11.25,-5.75)(12.25,-5.75)\ncput*{$r$} 
\multirput(16.4, 0.00)(0.3,0.0){3}{$\boldsymbol{\cdot}$} 
\multirput(16.1,-3.00)(0.3,0.0){3}{$\boldsymbol{\cdot}$}
\multirput(17.00,-4.50)(0.3,0.0){3}{$\boldsymbol{\cdot}$} 

\rput(13.40,1.25){$f_{m,p,i+1,0}$} 
\psline[linewidth=0.015,linestyle=dashed,linecolor=black](12.25,-5.25)(12.25,-3.75) 
\psline[linewidth=0.015,linestyle=dashed,linecolor=black](14.55,-5.25)(14.55, 0.75) 
\psline[linewidth=0.015,linestyle=dashed,linecolor=black](12.25, 0.75)(14.55, 0.75) 
\psline[linewidth=0.015,linestyle=dashed,linecolor=black](12.25,-5.25)(14.55,-5.25) 
\pcline[linewidth=0.015cm,linestyle=dotted](13.55,-3.75)(13.55,0.75) 
\pcline[linewidth=0.015cm,linestyle=dotted](12.25,-3.75)(14.55,-3.75) 
\pcline[linewidth=0.015cm,linestyle=dashed]{<->}(13.55,-2.25)(14.55,-2.25)\ncput*{$r$} 
\pcline[linewidth=0.015,linestyle=dashed]{<->}(12.25,-5.75)(14.55,-5.75)\ncput*{$m$} 
\pcline[linewidth=0.015,linestyle=solid]{<-}(12.25,-5.35)(12.25, -6.25)
\rput(12.75,-6.50){$(i+1)p$} 

\multirput(-1.5,-0.6)(0.3,-0.4){6}{$\boldsymbol{\cdot}$} 

\end{pspicture}
}
\\\hline
\end{tabular}

\caption{Relation between blocks}%
\label{block-dia}
\end{figure}
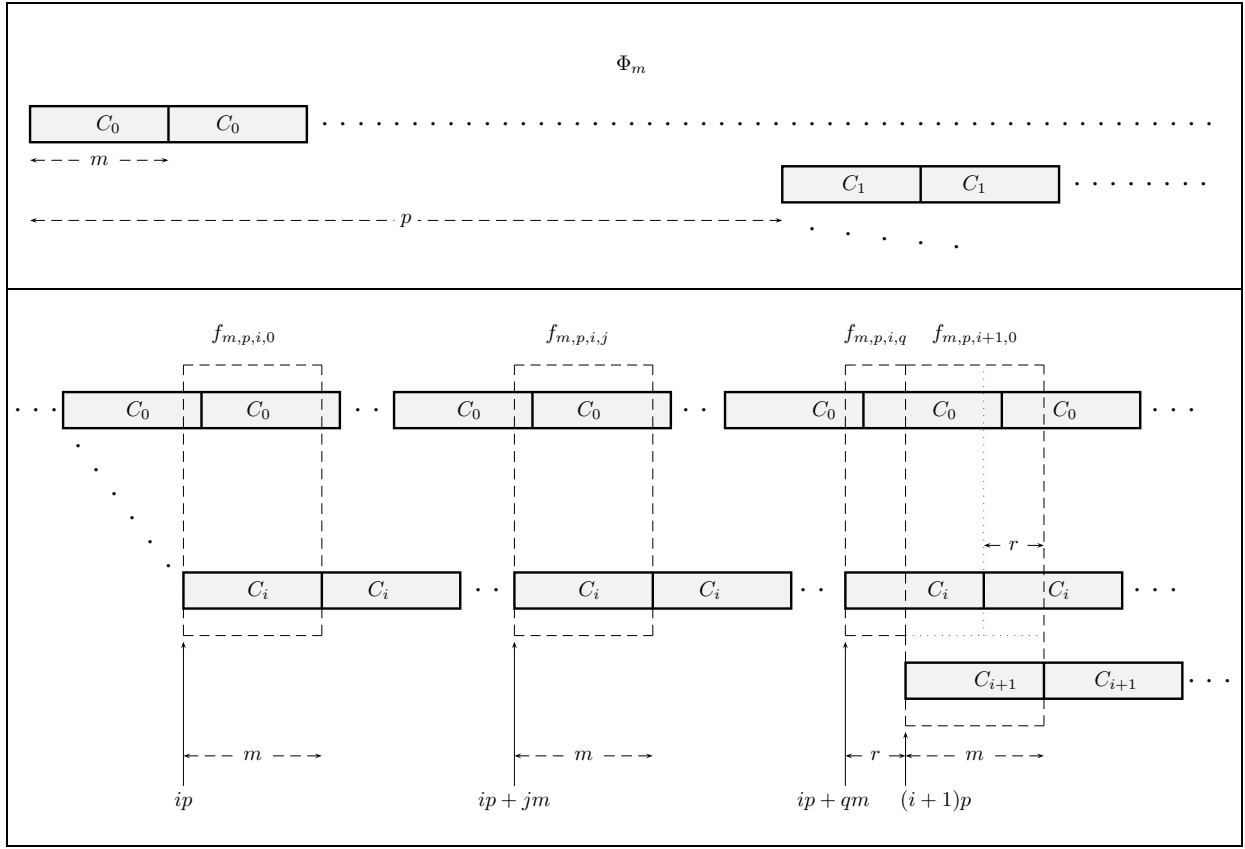

\begin{proof}
We begin by representing $\Phi_{mp}$ as a formal power series, as usual:%
\begin{equation}
\Phi_{mp}=\frac{\Phi_{m}\left(  x^{p}\right)  }{\Phi_{m}\left(  x\right)
}=\Phi_{m}\left(  x^{p}\right)  \ \frac{1}{x^{m}-1}\ \Psi_{m}=-\sum_{s\geq
0}a_{s}x^{sp}\sum_{u\geq0}x^{um}\Psi_{m}=\sum_{s\geq0}x^{sp}\sum_{u\geq
0}x^{um}C_{s} \label{power series}%
\end{equation}
where $C_{s}=-a_{s}\Psi_{m}$. Thus $\Phi_{mp}$ is the $p$-shifted sum of the
$m$-shifted sum of $C_{s}$. It is illustrated by the first diagram in Figure
\ref{block-dia}. Now we are ready to prove each claim one by one. For this we
will repeatedly refer to the second diagram in Figure \ref{block-dia}. Note
that the $m$-block $f_{m,p,i,0}$ can be obtained from $\Phi_{mp}$ by taking
the slice of length $m$ from $ip$. It is illustrated by the first block. 

\begin{enumerate}
\item $f_{m,p,i,0}=f_{m,p,i,j}\ \ $for $0\leq j\leq q-1$

The $m$-block $f_{m,p,i,j}\ \,$is obtained by taking the slice of length $m$
from $ip+jm$. It is illustrated by the second block. Comparing the two blocks
$f_{m,p,i,0}\ $and $f_{m,p,i,j}$, we observe that $f_{m,p,i,0}=f_{m,p,i,j}$,
that ism, $f_{m,p,i,j}$ does not depend on $j$.

\item $f_{m,p,i,q}=\mathcal{T}_{r}f_{m,p,i,0}$

The $r$-block $f_{m,p,i,q}$ is obtained by taking the slice of length $r$ from
$ip+qm$. It is illustrated by the third block. Comparing the two blocks
$f_{m,p,i,0}\ $and $f_{m,p,i,q}$, we observe that $f_{m,p,i,q}$ consists of
the first~$r$ terms of $f_{m,p,i,0}$, that is, $f_{m,p,i,q}=\mathcal{T}%
_{r}\ f_{m,p,i,0}$.

\item $f_{m,p,i+1,0}=\mathcal{R}_{m,r}f_{m,p,i,0}-a_{i+1}\Psi_{m}$

The $m$-block $f_{m,p,\left(  i+1\right)  ,0}\ \,$is obtained by taking the
slice of length $m$ from $\left(  i+1\right)  p$. It is illustrated by the
fourth block. Comparing the two blocks $f_{m,p,i,0}\ $and $f_{m,p,i+1,0}$, we
observe the followings.

\begin{enumerate}
\item The $m$-block $f_{m,p,i,0}\ $consists of $i+1$ rows and the $m$-block
$f_{m,p,i+1,0}$ consists of $i+2$ rows.

\item The sum of the first $i+1$ rows of $f_{m,p,i+1,0}$ is the rotation of
$f_{m,p,i,0}$ by $r$ modulo $m$. This happens because $\operatorname*{rem}%
(p,m)=r$.

\item The last row of $f_{m,p,i+1,0}$ is $C_{i+1}$
\end{enumerate}

Putting together, we have $f_{m,p,i+1,0}=\mathcal{R}_{m,r}f_{m,p,i,0}%
+C_{i+1}=\mathcal{R}_{m,r}f_{m,p,i,0}-a_{i+1}\Psi_{m}$.
\end{enumerate}
\end{proof}

\begin{theorem}
[Explicit expression for blocks]\label{explicit}We have%
\[
f_{m,p,i,0}=\Theta_{i}\Psi_{m}%
\]
\noindent where $\Theta_{i}$ is defined by
\[
\Theta_{i}=\mathrm{rem}\left(  w_{i}(x^{m-r}),\Phi_{m}\right)
\ \ \ \ \ \ \ w_{i}(x)=-\sum\limits_{0\leq s\leq i}a_{s}x^{i-s}%
\ \ \ \ \ \ \ \Phi_{m}=\sum\limits_{s}a_{s}x^{s}%
\]

\end{theorem}

\begin{proof}
We prove the claim by induction on $i$. When $i=0$. The claim is immediate
from%
\begin{align*}
f_{m,p,0,j}  &  =-a_{0}\Psi_{m} &  &  \text{from \eqref{power series}}\\
\Theta_{0}\Psi_{m}  &  =\operatorname*{rem}\left(  w_{0}(x^{m-r}),\Psi
_{m}\right)  \Psi_{m}=-a_{0}\Psi_{m} &
\end{align*}
Assume that the claim is true up to $i-1$. It suffices to show the claim holds
for $i$. Note
\begin{align*}
f_{m,p,i,j}  &  =\mathcal{R}_{m,r}f_{m,p,i-1,j}-a_{i}\Psi_{m} &  &  \text{from
Lemma \ref{relation}-3}\\
&  =\operatorname*{rem}\left(  x^{m-r}f_{m,p,i-1,j},~x^{m}-1\right)
-a_{i}\Psi_{m} &  &  \text{from Notation \ref{operation}}\\
&  =\operatorname*{rem}\left(  x^{m-r}\Theta_{i-1}\Psi_{m},~x^{m}-1\right)
-a_{i}\Psi_{m} &  &  \text{from the induction hypothesis}\\
&  =\operatorname*{rem}\left(  x^{m-r}\Theta_{i-1}\Psi_{m},~\Phi_{m}\Psi
_{m}\right)  -a_{i}\Psi_{m} &  &  \text{since }x^{m}-1=\Phi_{m}\Psi_{m}\\
&  =\left(  \operatorname*{rem}\left(  x^{m-r}\Theta_{i-1},~\Phi_{m}\right)
-a_{i}\right)  \Psi_{m} &  &  \text{by factoring out }\Psi_{m}\\
&  =\operatorname*{rem}\left(  x^{m-r}\Theta_{i-1}-a_{i},~\Phi_{m}\right)
\Psi_{m} &  &  \text{since }\deg\Phi_{m}>\deg a_{i}\\
&  =\operatorname*{rem}\left(  -x^{m-r}\sum_{s=0}^{i-1}a_{s}x^{(m-r)(i-1-s)}%
-a_{i},~\Phi_{m}\right)  \Psi_{m} &  &  \text{from the definition of }%
\Theta_{i-1}\\
&  =\operatorname*{rem}\left(  -\sum_{s=0}^{i-1}a_{s}x^{(m-r)(i-s)}%
-a_{i},~\Phi_{m}\right)  \Psi_{m} &  &  \text{by pushing }x^{m-r}\ \text{into
the summation}\\
&  =\operatorname*{rem}\left(  -\sum_{s=0}^{i}a_{s}x^{(m-r)(i-s)},~\Phi
_{m}\right)  \Psi_{m} &  &  \text{since }a_{i}=a_{i}x^{\left(  m-r\right)
\left(  i-i\right)  }\\
&  =\Theta_{i}\Psi_{m} &  &  \text{from the definition of }\Theta_{i}%
\end{align*}
as desired
\end{proof}

\noindent From the explicit expression of $m$-blocks, we can extract two
immediate but useful properties of~$m$-blocks.

\begin{corollary}
[Non-zero]\label{fmpi0<>0}$f_{m,p,i,0}\neq0$
\end{corollary}

\begin{proof}
Suppose that $f_{m,p,i,0}=0$ for some $i$. Since $\Psi_{m}\neq0$, from Theorem
\ref{explicit}, we have%
\[
\Phi_{m}\left(  x\right)  \ |\ w_{i}(x^{m-r})
\]
Let $\zeta$ be an $m$-th primitive root of unity. Then we have $w_{i}%
(\zeta^{m-r})=0$, because $m$ and $m-r$ are relatively prime. Note that
$\zeta^{m-r}$ is also an $m$-th primitive root of unity and the degree of its
minimal polynomial is $\varphi(m)$. However, $w_{i}\left(  x\right)
=a_{0}x^{i}+\cdots\neq0$ and $\deg w_{i}\left(  x\right)  $ $<\varphi(m)$. Contradiction.
\end{proof}

\begin{corollary}
[Degree difference]\label{degdiff}$\deg f_{m,p,i,0}-\operatorname*{tdeg}%
f_{m,p,i,0}\geq\psi\left(  m\right)  $ and the equality holds when $i=0$. 
\end{corollary}

\begin{proof}
Note
\begin{align*}
&  \ \deg f_{m,p,i,0}-\operatorname*{tdeg}f_{m,p,i,0} &  & \\
=  &  \ \deg\left(  \Theta_{i}\Psi_{m}\right)  -\operatorname*{tdeg}\left(
\Theta_{i}\Psi_{m}\right)  &  &  \text{from Theorem \ref{explicit}}\\
=  &  \ \left(  \deg\Theta_{i}+\deg\Psi_{m}\right)  -\left(
\operatorname*{tdeg}\Theta_{i}+\operatorname*{tdeg}\Psi_{m}\right)  &  &
\text{by expanding}\\
=  &  \ \deg\Theta_{i}+\psi\left(  m\right)  -\left(  \operatorname*{tdeg}%
\Theta_{i}+0\right)  &  &  \text{from $\operatorname*{tdeg}\Psi_{m}=0$}\\
=  &  \ \deg\Theta_{i}-\operatorname*{tdeg}\Theta_{i}+\psi\left(  m\right)  &
&  \text{by simplifying}%
\end{align*}
Hence
\begin{align*}
&  \deg f_{m,p,i,0}-\operatorname*{tdeg}f_{m,p,i,0}\ \ \geq\ \ \psi\left(
m\right)  &  &  \text{since }\operatorname*{tdeg}\Theta_{i}\leq\deg\Theta
_{i}\\
&  \deg f_{m,p,i,0}-\operatorname*{tdeg}f_{m,p,i,0}\ \ =\ \ \psi\left(
m\right)  &  &  \text{since }\Theta_{0}=-1
\end{align*}

\end{proof}

\section{Gaps}
\label{sec:gaps}
In this section, 
we will prove the main result of this paper (Theorem~\ref{main}) utilizing the block structures found in the previous section (Section~\ref{sec:blocks}).
For this, let us first recall the precise  definition of maximum gap~\cite{HLLP}.
\begin{definition}[Maximum gap]
\label{def:maxgap} Let $f=a_1x^{e_1}+a_2x^{e_2}+\cdots+a_nx^{e_n}$, where $a_1, \cdots,a_n\neq 0$ and\ $e_1<\cdots<e_n$. Then the maximum gap of $f$, written as $g(f)$ is given by\[g(f)=\max_{0 \leq i<n}e_{i+1}-e_i\]
\end{definition}

\noindent  Lemma \ref{relation}-1, Theorem~\ref{explicit} and
Corollary~\ref{fmpi0<>0} motivate the following notations of the gaps within
and between the blocks in~$\Phi_{mp}$. \textcolor{red}{}

\begin{notation}
[Block gaps]\label{block gaps}Let%

\begin{tabular}
[c]{llll}%
$g_{m,i}^{\mathsf{w}}$ & : & \emph{maximum gap within $f_{m,p,i,0}$} &
\emph{(the first $m$-block in $i$-th $p$-block)}\\
$g_{r,i}^{\mathsf{w}}$ & : & \emph{maximum gap within $f_{m,p,i,q}$} &
\emph{(the $r$-block in $i$-th $p$-block)}\\
$g_{m,i}^{\mathsf{b}}$ & : & \emph{gap between $f_{m,p,i,0}$ and $f_{m,p,i,1}%
$} & \emph{(the first two $m$-blocks within the $i$-th $p$-block)}\\
$g_{r,i}^{\mathsf{b}}$ & : & \emph{gap between $f_{m,p,i,q-1}$ and
$f_{m,p,i,q}$} & \emph{(the last $m$-block and the $r$-block within the $i$-th
$p$-block)}\\
$g_{p,i}^{\mathsf{b}}$ & : & \emph{gap between $f_{m,p,i-1}$ and $f_{m,p,i}$}
& \emph{(the $i$-th $p$-block and the $(i-1)$-th $p$-block)}%
\end{tabular}

\medskip

\noindent In some cases, the above quantities are meaningless. In such cases,
for the sake of notational simplicity, we will assign $0$ to them, that is,%

\begin{tabular}
[c]{lll}%
$g_{r,i}^{\mathsf{w}}=0$ & if & $f_{m,p,i,q}=0$\\
$g_{m,i}^{\mathsf{b}}=0$ & if & $q=1$\\
$g_{r,i}^{\mathsf{b}}=0$ & if & $f_{m,p,i,q}=0$\\
$g_{p,i}^{\mathsf{b}}=0$ & if & \thinspace$i=0$%
\end{tabular}

\end{notation}

\begin{example}
We illustrate the notations for $m=15$ and $p=53$ as in
Example~\ref{example_block} . The following diagram illustrates the division
of $\Phi_{mp}$ from the $p$-block when $i=2$ to the beginning of the $p$-block
when $i=4$. The dark bar stands for non-zero term and the white bar stands for
zero term.%

\[
\noindent \scalebox{0.75}{
\begin{pspicture}(-0.5,1.75)(22,4.5)

\definecolor{MyColor}{rgb}{0.95,0.95,0.95} 

\psframe[linewidth=0.015,linestyle=none,dimen=middle,fillstyle=solid,fillcolor=MyColor](0.00,3.00)(22.00,3.50)
\pcline[linewidth=0.015,linestyle=dashed]{->}(-0.40,4.50)(0.00,4.50)
\pcline[linewidth=0.015,linestyle=dashed]{->}(-0.40,4.00)(0.00,4.00)
\psframe[linewidth=0.015,linestyle=none,dimen=middle,fillstyle=solid,fillcolor=MyColor](-0.40,3.00)(0.00,3.50)
\rput(0.00,5.00){$i=2$}
\pcline[linewidth=0.015,linestyle=dashed]{<->}(0.00,4.50)(10.60,4.50)\ncput*{$p$}
\psline[linewidth=0.015](0.00,1.50)(0.00,4.80)
\pcline[linewidth=0.015,linestyle=dashed]{<->}(0.00,4.00)(3.00,4.00)\ncput*{$m$}
\pcline[linewidth=0.015,linestyle=dashed]{<->}(3.00,4.00)(6.00,4.00)\ncput*{$m$}
\psline[linewidth=0.015](3.00,2.50)(3.00,4.30)
\pcline[linewidth=0.015,linestyle=dashed]{<->}(6.00,4.00)(9.00,4.00)\ncput*{$m$}
\psline[linewidth=0.015](6.00,2.50)(6.00,4.30)
\pcline[linewidth=0.015,linestyle=dashed]{<->}(9.00,4.00)(10.60,4.00)\ncput*{$r$}
\psline[linewidth=0.015](9.00,2.50)(9.00,4.30)
\rput(10.60,5.00){$i=3$}
\pcline[linewidth=0.015,linestyle=dashed]{<->}(10.60,4.50)(21.20,4.50)\ncput*{$p$}
\psline[linewidth=0.015](10.60,1.50)(10.60,4.80)
\pcline[linewidth=0.015,linestyle=dashed]{<->}(10.60,4.00)(13.60,4.00)\ncput*{$m$}
\pcline[linewidth=0.015,linestyle=dashed]{<->}(13.60,4.00)(16.60,4.00)\ncput*{$m$}
\psline[linewidth=0.015](13.60,2.50)(13.60,4.30)
\pcline[linewidth=0.015,linestyle=dashed]{<->}(16.60,4.00)(19.60,4.00)\ncput*{$m$}
\psline[linewidth=0.015](16.60,2.50)(16.60,4.30)
\pcline[linewidth=0.015,linestyle=dashed]{<->}(19.60,4.00)(21.20,4.00)\ncput*{$r$}
\psline[linewidth=0.015](19.60,2.50)(19.60,4.30)
\psline[linewidth=0.015](21.20,1.50)(21.20,4.80)
\pcline[linewidth=0.015,linestyle=dashed]{<-}(21.20,4.50)(22.00,4.50)
\pcline[linewidth=0.015,linestyle=dashed]{<-}(21.20,4.00)(22.00,4.00)
\rput(21.20,5.00){$i=4$}
\psframe[linewidth=0.015,linestyle=solid,dimen=middle,fillstyle=solid,fillcolor=darkgray](0.00,3.00)(0.20,3.50)
\psframe[linewidth=0.015,linestyle=solid,dimen=middle,fillstyle=solid,fillcolor=darkgray](0.20,3.00)(0.40,3.50)
\psframe[linewidth=0.015,linestyle=solid,dimen=middle,fillstyle=solid,fillcolor=white](0.40,3.00)(0.60,3.50)
\psframe[linewidth=0.015,linestyle=solid,dimen=middle,fillstyle=solid,fillcolor=white](0.60,3.00)(0.80,3.50)
\psframe[linewidth=0.015,linestyle=solid,dimen=middle,fillstyle=solid,fillcolor=darkgray](0.80,3.00)(1.00,3.50)
\psframe[linewidth=0.015,linestyle=solid,dimen=middle,fillstyle=solid,fillcolor=darkgray](1.00,3.00)(1.20,3.50)
\psframe[linewidth=0.015,linestyle=solid,dimen=middle,fillstyle=solid,fillcolor=darkgray](1.20,3.00)(1.40,3.50)
\psframe[linewidth=0.015,linestyle=solid,dimen=middle,fillstyle=solid,fillcolor=darkgray](1.40,3.00)(1.60,3.50)
\psframe[linewidth=0.015,linestyle=solid,dimen=middle,fillstyle=solid,fillcolor=darkgray](1.60,3.00)(1.80,3.50)
\psframe[linewidth=0.015,linestyle=solid,dimen=middle,fillstyle=solid,fillcolor=darkgray](1.80,3.00)(2.00,3.50)
\psframe[linewidth=0.015,linestyle=solid,dimen=middle,fillstyle=solid,fillcolor=white](2.00,3.00)(2.20,3.50)
\psframe[linewidth=0.015,linestyle=solid,dimen=middle,fillstyle=solid,fillcolor=white](2.20,3.00)(2.40,3.50)
\psframe[linewidth=0.015,linestyle=solid,dimen=middle,fillstyle=solid,fillcolor=darkgray](2.40,3.00)(2.60,3.50)
\psframe[linewidth=0.015,linestyle=solid,dimen=middle,fillstyle=solid,fillcolor=darkgray](2.60,3.00)(2.80,3.50)
\psframe[linewidth=0.015,linestyle=solid,dimen=middle,fillstyle=solid,fillcolor=darkgray](2.80,3.00)(3.00,3.50)
\psframe[linewidth=0.015,linestyle=solid,dimen=middle,fillstyle=solid,fillcolor=darkgray](3.00,3.00)(3.20,3.50)
\psframe[linewidth=0.015,linestyle=solid,dimen=middle,fillstyle=solid,fillcolor=darkgray](3.20,3.00)(3.40,3.50)
\psframe[linewidth=0.015,linestyle=solid,dimen=middle,fillstyle=solid,fillcolor=white](3.40,3.00)(3.60,3.50)
\psframe[linewidth=0.015,linestyle=solid,dimen=middle,fillstyle=solid,fillcolor=white](3.60,3.00)(3.80,3.50)
\psframe[linewidth=0.015,linestyle=solid,dimen=middle,fillstyle=solid,fillcolor=darkgray](3.80,3.00)(4.00,3.50)
\psframe[linewidth=0.015,linestyle=solid,dimen=middle,fillstyle=solid,fillcolor=darkgray](4.00,3.00)(4.20,3.50)
\psframe[linewidth=0.015,linestyle=solid,dimen=middle,fillstyle=solid,fillcolor=darkgray](4.20,3.00)(4.40,3.50)
\psframe[linewidth=0.015,linestyle=solid,dimen=middle,fillstyle=solid,fillcolor=darkgray](4.40,3.00)(4.60,3.50)
\psframe[linewidth=0.015,linestyle=solid,dimen=middle,fillstyle=solid,fillcolor=darkgray](4.60,3.00)(4.80,3.50)
\psframe[linewidth=0.015,linestyle=solid,dimen=middle,fillstyle=solid,fillcolor=darkgray](4.80,3.00)(5.00,3.50)
\psframe[linewidth=0.015,linestyle=solid,dimen=middle,fillstyle=solid,fillcolor=white](5.00,3.00)(5.20,3.50)
\psframe[linewidth=0.015,linestyle=solid,dimen=middle,fillstyle=solid,fillcolor=white](5.20,3.00)(5.40,3.50)
\psframe[linewidth=0.015,linestyle=solid,dimen=middle,fillstyle=solid,fillcolor=darkgray](5.40,3.00)(5.60,3.50)
\psframe[linewidth=0.015,linestyle=solid,dimen=middle,fillstyle=solid,fillcolor=darkgray](5.60,3.00)(5.80,3.50)
\psframe[linewidth=0.015,linestyle=solid,dimen=middle,fillstyle=solid,fillcolor=darkgray](5.80,3.00)(6.00,3.50)
\psframe[linewidth=0.015,linestyle=solid,dimen=middle,fillstyle=solid,fillcolor=darkgray](6.00,3.00)(6.20,3.50)
\psframe[linewidth=0.015,linestyle=solid,dimen=middle,fillstyle=solid,fillcolor=darkgray](6.20,3.00)(6.40,3.50)
\psframe[linewidth=0.015,linestyle=solid,dimen=middle,fillstyle=solid,fillcolor=white](6.40,3.00)(6.60,3.50)
\psframe[linewidth=0.015,linestyle=solid,dimen=middle,fillstyle=solid,fillcolor=white](6.60,3.00)(6.80,3.50)
\psframe[linewidth=0.015,linestyle=solid,dimen=middle,fillstyle=solid,fillcolor=darkgray](6.80,3.00)(7.00,3.50)
\psframe[linewidth=0.015,linestyle=solid,dimen=middle,fillstyle=solid,fillcolor=darkgray](7.00,3.00)(7.20,3.50)
\psframe[linewidth=0.015,linestyle=solid,dimen=middle,fillstyle=solid,fillcolor=darkgray](7.20,3.00)(7.40,3.50)
\psframe[linewidth=0.015,linestyle=solid,dimen=middle,fillstyle=solid,fillcolor=darkgray](7.40,3.00)(7.60,3.50)
\psframe[linewidth=0.015,linestyle=solid,dimen=middle,fillstyle=solid,fillcolor=darkgray](7.60,3.00)(7.80,3.50)
\psframe[linewidth=0.015,linestyle=solid,dimen=middle,fillstyle=solid,fillcolor=darkgray](7.80,3.00)(8.00,3.50)
\psframe[linewidth=0.015,linestyle=solid,dimen=middle,fillstyle=solid,fillcolor=white](8.00,3.00)(8.20,3.50)
\psframe[linewidth=0.015,linestyle=solid,dimen=middle,fillstyle=solid,fillcolor=white](8.20,3.00)(8.40,3.50)
\psframe[linewidth=0.015,linestyle=solid,dimen=middle,fillstyle=solid,fillcolor=darkgray](8.40,3.00)(8.60,3.50)
\psframe[linewidth=0.015,linestyle=solid,dimen=middle,fillstyle=solid,fillcolor=darkgray](8.60,3.00)(8.80,3.50)
\psframe[linewidth=0.015,linestyle=solid,dimen=middle,fillstyle=solid,fillcolor=darkgray](8.80,3.00)(9.00,3.50)
\psframe[linewidth=0.015,linestyle=solid,dimen=middle,fillstyle=solid,fillcolor=darkgray](9.00,3.00)(9.20,3.50)
\psframe[linewidth=0.015,linestyle=solid,dimen=middle,fillstyle=solid,fillcolor=darkgray](9.20,3.00)(9.40,3.50)
\psframe[linewidth=0.015,linestyle=solid,dimen=middle,fillstyle=solid,fillcolor=white](9.40,3.00)(9.60,3.50)
\psframe[linewidth=0.015,linestyle=solid,dimen=middle,fillstyle=solid,fillcolor=white](9.60,3.00)(9.80,3.50)
\psframe[linewidth=0.015,linestyle=solid,dimen=middle,fillstyle=solid,fillcolor=darkgray](9.80,3.00)(10.00,3.50)
\psframe[linewidth=0.015,linestyle=solid,dimen=middle,fillstyle=solid,fillcolor=darkgray](10.00,3.00)(10.20,3.50)
\psframe[linewidth=0.015,linestyle=solid,dimen=middle,fillstyle=solid,fillcolor=darkgray](10.20,3.00)(10.40,3.50)
\psframe[linewidth=0.015,linestyle=solid,dimen=middle,fillstyle=solid,fillcolor=darkgray](10.40,3.00)(10.60,3.50)
\psframe[linewidth=0.015,linestyle=solid,dimen=middle,fillstyle=solid,fillcolor=white](10.60,3.00)(10.80,3.50)
\psframe[linewidth=0.015,linestyle=solid,dimen=middle,fillstyle=solid,fillcolor=white](10.80,3.00)(11.00,3.50)
\psframe[linewidth=0.015,linestyle=solid,dimen=middle,fillstyle=solid,fillcolor=darkgray](11.00,3.00)(11.20,3.50)
\psframe[linewidth=0.015,linestyle=solid,dimen=middle,fillstyle=solid,fillcolor=white](11.20,3.00)(11.40,3.50)
\psframe[linewidth=0.015,linestyle=solid,dimen=middle,fillstyle=solid,fillcolor=darkgray](11.40,3.00)(11.60,3.50)
\psframe[linewidth=0.015,linestyle=solid,dimen=middle,fillstyle=solid,fillcolor=white](11.60,3.00)(11.80,3.50)
\psframe[linewidth=0.015,linestyle=solid,dimen=middle,fillstyle=solid,fillcolor=darkgray](11.80,3.00)(12.00,3.50)
\psframe[linewidth=0.015,linestyle=solid,dimen=middle,fillstyle=solid,fillcolor=white](12.00,3.00)(12.20,3.50)
\psframe[linewidth=0.015,linestyle=solid,dimen=middle,fillstyle=solid,fillcolor=darkgray](12.20,3.00)(12.40,3.50)
\psframe[linewidth=0.015,linestyle=solid,dimen=middle,fillstyle=solid,fillcolor=white](12.40,3.00)(12.60,3.50)
\psframe[linewidth=0.015,linestyle=solid,dimen=middle,fillstyle=solid,fillcolor=white](12.60,3.00)(12.80,3.50)
\psframe[linewidth=0.015,linestyle=solid,dimen=middle,fillstyle=solid,fillcolor=darkgray](12.80,3.00)(13.00,3.50)
\psframe[linewidth=0.015,linestyle=solid,dimen=middle,fillstyle=solid,fillcolor=darkgray](13.00,3.00)(13.20,3.50)
\psframe[linewidth=0.015,linestyle=solid,dimen=middle,fillstyle=solid,fillcolor=darkgray](13.20,3.00)(13.40,3.50)
\psframe[linewidth=0.015,linestyle=solid,dimen=middle,fillstyle=solid,fillcolor=darkgray](13.40,3.00)(13.60,3.50)
\psframe[linewidth=0.015,linestyle=solid,dimen=middle,fillstyle=solid,fillcolor=white](13.60,3.00)(13.80,3.50)
\psframe[linewidth=0.015,linestyle=solid,dimen=middle,fillstyle=solid,fillcolor=white](13.80,3.00)(14.00,3.50)
\psframe[linewidth=0.015,linestyle=solid,dimen=middle,fillstyle=solid,fillcolor=darkgray](14.00,3.00)(14.20,3.50)
\psframe[linewidth=0.015,linestyle=solid,dimen=middle,fillstyle=solid,fillcolor=white](14.20,3.00)(14.40,3.50)
\psframe[linewidth=0.015,linestyle=solid,dimen=middle,fillstyle=solid,fillcolor=darkgray](14.40,3.00)(14.60,3.50)
\psframe[linewidth=0.015,linestyle=solid,dimen=middle,fillstyle=solid,fillcolor=white](14.60,3.00)(14.80,3.50)
\psframe[linewidth=0.015,linestyle=solid,dimen=middle,fillstyle=solid,fillcolor=darkgray](14.80,3.00)(15.00,3.50)
\psframe[linewidth=0.015,linestyle=solid,dimen=middle,fillstyle=solid,fillcolor=white](15.00,3.00)(15.20,3.50)
\psframe[linewidth=0.015,linestyle=solid,dimen=middle,fillstyle=solid,fillcolor=darkgray](15.20,3.00)(15.40,3.50)
\psframe[linewidth=0.015,linestyle=solid,dimen=middle,fillstyle=solid,fillcolor=white](15.40,3.00)(15.60,3.50)
\psframe[linewidth=0.015,linestyle=solid,dimen=middle,fillstyle=solid,fillcolor=white](15.60,3.00)(15.80,3.50)
\psframe[linewidth=0.015,linestyle=solid,dimen=middle,fillstyle=solid,fillcolor=darkgray](15.80,3.00)(16.00,3.50)
\psframe[linewidth=0.015,linestyle=solid,dimen=middle,fillstyle=solid,fillcolor=darkgray](16.00,3.00)(16.20,3.50)
\psframe[linewidth=0.015,linestyle=solid,dimen=middle,fillstyle=solid,fillcolor=darkgray](16.20,3.00)(16.40,3.50)
\psframe[linewidth=0.015,linestyle=solid,dimen=middle,fillstyle=solid,fillcolor=darkgray](16.40,3.00)(16.60,3.50)
\psframe[linewidth=0.015,linestyle=solid,dimen=middle,fillstyle=solid,fillcolor=white](16.60,3.00)(16.80,3.50)
\psframe[linewidth=0.015,linestyle=solid,dimen=middle,fillstyle=solid,fillcolor=white](16.80,3.00)(17.00,3.50)
\psframe[linewidth=0.015,linestyle=solid,dimen=middle,fillstyle=solid,fillcolor=darkgray](17.00,3.00)(17.20,3.50)
\psframe[linewidth=0.015,linestyle=solid,dimen=middle,fillstyle=solid,fillcolor=white](17.20,3.00)(17.40,3.50)
\psframe[linewidth=0.015,linestyle=solid,dimen=middle,fillstyle=solid,fillcolor=darkgray](17.40,3.00)(17.60,3.50)
\psframe[linewidth=0.015,linestyle=solid,dimen=middle,fillstyle=solid,fillcolor=white](17.60,3.00)(17.80,3.50)
\psframe[linewidth=0.015,linestyle=solid,dimen=middle,fillstyle=solid,fillcolor=darkgray](17.80,3.00)(18.00,3.50)
\psframe[linewidth=0.015,linestyle=solid,dimen=middle,fillstyle=solid,fillcolor=white](18.00,3.00)(18.20,3.50)
\psframe[linewidth=0.015,linestyle=solid,dimen=middle,fillstyle=solid,fillcolor=darkgray](18.20,3.00)(18.40,3.50)
\psframe[linewidth=0.015,linestyle=solid,dimen=middle,fillstyle=solid,fillcolor=white](18.40,3.00)(18.60,3.50)
\psframe[linewidth=0.015,linestyle=solid,dimen=middle,fillstyle=solid,fillcolor=white](18.60,3.00)(18.80,3.50)
\psframe[linewidth=0.015,linestyle=solid,dimen=middle,fillstyle=solid,fillcolor=darkgray](18.80,3.00)(19.00,3.50)
\psframe[linewidth=0.015,linestyle=solid,dimen=middle,fillstyle=solid,fillcolor=darkgray](19.00,3.00)(19.20,3.50)
\psframe[linewidth=0.015,linestyle=solid,dimen=middle,fillstyle=solid,fillcolor=darkgray](19.20,3.00)(19.40,3.50)
\psframe[linewidth=0.015,linestyle=solid,dimen=middle,fillstyle=solid,fillcolor=darkgray](19.40,3.00)(19.60,3.50)
\psframe[linewidth=0.015,linestyle=solid,dimen=middle,fillstyle=solid,fillcolor=white](19.60,3.00)(19.80,3.50)
\psframe[linewidth=0.015,linestyle=solid,dimen=middle,fillstyle=solid,fillcolor=white](19.80,3.00)(20.00,3.50)
\psframe[linewidth=0.015,linestyle=solid,dimen=middle,fillstyle=solid,fillcolor=darkgray](20.00,3.00)(20.20,3.50)
\psframe[linewidth=0.015,linestyle=solid,dimen=middle,fillstyle=solid,fillcolor=white](20.20,3.00)(20.40,3.50)
\psframe[linewidth=0.015,linestyle=solid,dimen=middle,fillstyle=solid,fillcolor=darkgray](20.40,3.00)(20.60,3.50)
\psframe[linewidth=0.015,linestyle=solid,dimen=middle,fillstyle=solid,fillcolor=white](20.60,3.00)(20.80,3.50)
\psframe[linewidth=0.015,linestyle=solid,dimen=middle,fillstyle=solid,fillcolor=darkgray](20.80,3.00)(21.00,3.50)
\psframe[linewidth=0.015,linestyle=solid,dimen=middle,fillstyle=solid,fillcolor=white](21.00,3.00)(21.20,3.50)
\psframe[linewidth=0.015,linestyle=solid,dimen=middle,fillstyle=solid,fillcolor=white](21.20,3.00)(21.40,3.50)
\psframe[linewidth=0.015,linestyle=solid,dimen=middle,fillstyle=solid,fillcolor=darkgray](21.40,3.00)(21.60,3.50)
\psframe[linewidth=0.015,linestyle=solid,dimen=middle,fillstyle=solid,fillcolor=darkgray](21.60,3.00)(21.80,3.50)

\psline[linewidth=0.015](0.40,3.00)(0.40,2.50)
\psline[linewidth=0.015](0.80,3.00)(0.80,2.50)
\psline[linewidth=0.015,linestyle=solid]{<->}(0.40,2.80)(0.80,2.80)
\psline[linewidth=0.015,linestyle=solid]{<-}(0.60,2.70)(0.60,2.30)
\rput(0.60,2.00){$g_{m,2}^{\mathsf{w}}$}

\psline[linewidth=0.015](9.40,3.00)(9.40,2.50)
\psline[linewidth=0.015](9.80,3.00)(9.80,2.50)
\psline[linewidth=0.015,linestyle=solid]{<->}(9.40,2.80)(9.80,2.80)
\psline[linewidth=0.015,linestyle=solid]{<-}(9.60,2.70)(9.60,2.30)
\rput(9.60,2.00){$g_{r,2}^{\mathsf{w}}$}

\psline[linewidth=0.015](11.00,3.0)(11.00,2.50)
\psline[linewidth=0.015,linestyle=solid]{<->}(10.60,2.80)(11.00,2.80)
\psline[linewidth=0.015,linestyle=solid]{<-}(10.80,2.70)(10.80,2.30)
\rput(11,2.00){$g_{p,3}^{\mathsf{b}}$}

\psline[linewidth=0.015](14.00,3.00)(14.00,2.50)
\psline[linewidth=0.015,linestyle=solid]{<->}(13.60,2.80)(14.00,2.80)
\psline[linewidth=0.015,linestyle=solid]{<-}(13.80,2.70)(13.80,2.30)
\rput(13.9,2.00){$g_{m,3}^{\mathsf{b}}$}

\psframe[linestyle=solid,dimen=outer,fillstyle=solid,fillcolor=white](19.60,3.00)(19.80,3.50)
\psframe[linestyle=solid,dimen=outer,fillstyle=solid,fillcolor=white](19.80,3.00)(20.00,3.50)
\psline[linewidth=0.015](20.00,3.00)(20.00,2.50)
\psline[linewidth=0.015,linestyle=solid]{<->}(19.60,2.80)(20.00,2.80)
\psline[linewidth=0.015,linestyle=solid]{<-}(19.80,2.70)(19.80,2.30)
\rput(19.8,2.00){$g_{r,3}^{\mathsf{b}}$}

\psline[linewidth=0.015](21.00,3.00)(21.00,2.50)
\psline[linewidth=0.015](21.40,3.00)(21.40,2.50)
\psline[linewidth=0.015,linestyle=solid]{<->}(21.00,2.80)(21.40,2.80)
\psline[linewidth=0.015,linestyle=solid]{<-}(21.25,2.70)(21.40,2.30)
\rput(21.6,2.00){$g_{p,4}^{\mathsf{b}}$}
\end{pspicture}
}

\]

\end{example}

\noindent Next we express the maximum gap of $\Phi_{mp}$ in terms of the block gaps.

\begin{lemma}
\label{g_in_blockgaps}$g\left(  \Phi_{mp}\right)  =\max\limits_{0\leq
i\leq\varphi\left(  m\right)  -1}\left\{  g_{m,i}^{\mathsf{w}},\ \ g_{r,i}%
^{\mathsf{w}},\ \ g_{m,i}^{\mathsf{b}},\ \ g_{r,i}^{\mathsf{b}},\ \ g_{p,i}%
^{\mathsf{b}}\right\}  $
\end{lemma}

\begin{proof}
Immediate from Lemma \ref{relation}-1, Theorem~\ref{explicit} and
Corollary~\ref{fmpi0<>0}.
\end{proof}

\noindent The above lemma tells us that we need to investigate the block gaps:
$g_{m,i}^{\mathsf{w}},\ g_{r,i}^{\mathsf{w}},\ g_{m,i}^{\mathsf{b}}%
,\ g_{r,i}^{\mathsf{b}},\ g_{p,i}^{\mathsf{b}}$. In particular, we need to
show that each block gap is less than or equal to $\varphi\left(  m\right)  $
and at least one of them is equal to $\varphi\left(  m\right)  $. In the
following, we will do so one by one.

\begin{lemma}
\label{gwm_i}$g_{m,i}^{\mathsf{w}}\leq\varphi\left(  m\right)  $
\end{lemma}

\begin{proof}
We will first prove a more general claim. Then the lemma will follow immediately.

\begin{enumerate}
\item Let $h=v\Psi_{m}$ such that $\deg h<m$. Then $g\left(  h\right)
\leq\varphi(m)$.

We prove the claim by contradiction.

\begin{enumerate}
\item Assume $g(h)>\varphi(m)$. Let
\[
h=\sum_{i=1}^{t}h_{i}\;x^{e_{i}}%
\]
where $0\leq e_{1}<e_{2}<\cdots<e_{t}<m\ \ $and $h_{1},\ldots,h_{t}\neq0$.
Then for some $k$ we have
\[
g(h)=e_{k}-e_{k-1}>\varphi(m).
\]

\item Note
\[
h\Phi_{m}=\sum_{i=1}^{t}h_{i}\;x^{e_{i}}\;\Phi_{m}(x)=\sum_{i=1}^{t}%
B_{i}\ \ \ \ \ \text{where }B_{i}=h_{i}x^{e_{i}}\Phi_{m}(x)
\]
The polynomials $B_{1},\ldots,B_{t}$ are visualized in the following diagram
\[
\scalebox{0.8}{
\begin{pspicture}(0,-4.2)(11,0.5)
\definecolor{MyColor}{rgb}{0.95,0.95,0.95}
\def\poly{
\psframe[linewidth=0.04,fillstyle=solid,fillcolor=MyColor](3.0,0.3)(0.0,-0.3)
\psframe[linewidth=0.04,fillstyle=solid,fillcolor=black](0.2,0.3)(0.0,-0.3)
\psframe[linewidth=0.04,fillstyle=solid,fillcolor=black](3.0,0.3)(2.8,-0.3)
}
\multips(0.00, 0.00)(2,-0.75){1}{\poly}
\multips(2.00,-1.25)(4,-0.75){2}{\poly}
\multips(8.00,-3.30)(2,-0.75){1}{\poly}
\multirput(1.85,-0.46)(0.30,-0.18){3}{$\boldsymbol{\cdot}$}
\multirput(8.00,-2.49)(0.30,-0.18){3}{$\boldsymbol{\cdot}$}
\psline[linewidth=0.02cm](4.9,-1.2)(4.9,-2.5)
\pcline[linewidth=0.02cm,linestyle=dashed]{<->}(2.1,-2.0)(4.9,-2.0)\ncput*{$\varphi(m)$}
\psline[linewidth=0.02cm](2.1,-1.25)(2.1,-4.0)
\psline[linewidth=0.02cm](6.1,-1.9)(6.1,-4.0)
\pcline[linewidth=0.02cm,linestyle=dashed]{<->}(2.1,-3.5)(6.1,-3.5)\ncput*{$g(h)$}
\rput( 1.50, 0.00){$B_1$}
\rput( 3.60,-1.25){$B_{k-1}$}
\rput( 7.50,-2.00){$B_{k}$}
\rput( 9.65,-3.30){$B_{t}$}
\rput( 2.10,-4.20){$e_{k-1}$}
\rput( 6.10,-4.20){$e_k$}
\end{pspicture}}
\]
where the black color indicates that the corresponding exponent (term) occurs
in the polynomial and the grey color indicates that we do not need to know
whether it occurs or not.

\medskip

Observe that the term $x^{e_{k}}$ appears only in $B_{k}$ and thus it does not
get cancelled during the summation. Hence%
\[
\operatorname*{coeff}\nolimits_{e_{k}}\left(  h\Phi_{m}\right)  \neq0
\]

\item Note%
\[
h\Phi_{m}=v\Psi_{m}\Phi_{m}=v\left(  x^{m}-1\right)  =vx^{m}-v
\]
Observe
\[
\deg v\;<\;\varphi\left(  m\right)  \;<\;e_{k}<\;m\;\leq\;\operatorname*{tdeg}%
\left(  vx^{m}\right)
\]
Hence%
\[
\operatorname*{coeff}\nolimits_{e_{k}}\left(  h\Phi_{m}\right)  =0
\]

\item Contradiction.
\end{enumerate}

\item $g_{m,i}^{\mathsf{w}}\leq\varphi(m)$.

Immediate from applying the above claim to the expression in Theorem
\ref{explicit}.
\end{enumerate}
\end{proof}

\begin{lemma}
\label{gwr_i}$g_{r,i}^{\mathsf{w}}\leq\varphi\left(  m\right)  . $
\end{lemma}

\begin{proof}
When $f_{m,p,i,q}=0$, the claim is vacuously true. When $f_{m,p,i,q}\neq0$,
the claim is immediate from
\begin{align*}
g_{r,i}^{\mathsf{w}}\ =  &  \ g\left(  f_{m,p,i,q}\right)  &  & \\
=  &  \ g\left(  \mathcal{T}_{r}f_{m,p,i,0}\right)  &  &  \text{from Lemma
\ref{relation}-2}\\
\leq &  \ g\left(  f_{m,p,i,0}\right)  &  &  \text{since truncation does not
increase gap}\\
\leq &  \ \varphi\left(  m\right)  &  &  \text{from Lemma \ref{gwm_i}}%
\end{align*}

\end{proof}

\begin{lemma}
\label{gbm_i}$g_{m,i}^{\mathsf{b}}\leq\varphi\left(  m\right)  $
\end{lemma}

\begin{proof}
When $q=1$, the claim is vacuously true. When $q>1$, the claim is immediate
from
\begin{align*}
g_{m,i}^{\mathsf{b}}\ =  &  \ m+\operatorname*{tdeg}f_{m,p,i,1}-\deg
f_{m,p,i,0} &  & \\
=  &  \ m+\operatorname*{tdeg}f_{m,p,i,0}-\deg f_{m,p,i,0} &  &  \text{from
Lemma\ref{relation}-1}\\
\leq &  \ \varphi\left(  m\right)  &  &  \text{from Corollary \ref{degdiff}}%
\end{align*}

\end{proof}

\begin{lemma}
\label{gbr_i}$g_{r,i}^{\mathsf{b}}\leq\varphi\left(  m\right)  $ and the
equality holds when $i=0$.
\end{lemma}

\begin{proof}
When $f_{m,p,i,q}=0$, the claim is vacuously true. When $f_{m,p,i,q}\neq0$,
note
\begin{align*}
g_{r,i}^{\mathsf{b}}\ =  &  \ m+\operatorname*{tdeg}f_{m,p,i,q}-\deg
f_{m,p,i,q-1} &  & \\
=  &  \ m+\operatorname*{tdeg}\mathcal{T}_{r}f_{m,p,i,0}-\deg f_{m,p,i,0} &
&  \text{from Lemma \ref{relation}-1,2}\\
=  &  \ m+\operatorname*{tdeg}f_{m,p,i,0}-\deg f_{m,p,i,0} &  &  \text{since
}f_{m,p,i,q}\neq0\\
\leq &  \ \varphi\left(  m\right)  &  &  \text{from Corollary \ref{degdiff}}%
\end{align*}

\end{proof}

\begin{lemma}
\label{gbp_i}$g_{p,i}^{\mathsf{b}}\leq\varphi\left(  m\right)  $
\end{lemma}

\begin{proof}
When $i=0$, the claim is vacuously true. Hence from now on we assume that
$i\geq1$. The proof is a bit long. Hence we will divide the proof into several
claims and their proofs.

\begin{enumerate}
\item $\deg f_{m,p,i,0}=\deg\mathcal{R}_{m,r}f_{m,p,i-1,0}$

\begin{enumerate}
\item $\Psi_{m}\ |\ f_{m,p,i,0}$ from Theorem \ref{explicit}.

\item $\Psi_{m}\ |\ \mathcal{R}_{m,r}f_{m,p,i-1,0}\ $from Lemma \ref{relation}%
-3 and the step above.

\item $\mathcal{R}_{m,r}f_{m,p,i-1,0}\neq0\ $from Corollary \ref{fmpi0<>0} and
Notation \ref{operation}.

\item $\psi\left(  m\right)  \leq\deg\mathcal{R}_{m,r}f_{m,p,i-1,0}\ \ $from
the two steps above.

\item $\psi\left(  m\right)  \leq\deg f_{m,p,i,0}\ \ $from Corollary
\ref{degdiff}.

\item $\deg f_{m,p,i,0}=\deg\mathcal{R}_{m,r}f_{m,p,i-1,0}\ $from Lemma
\ref{relation}-3.
\end{enumerate}

\item $\deg f_{m,p,i-1}=\deg f_{m,p,i,0}+p-m$

Note%
\begin{align*}
&  \deg f_{m,p,i-1} &  & \\
=  &  \
\begin{cases}
\deg f_{m,p,i-1,q}+qm & \text{if \thinspace}f_{m,p,i-1,q}\neq0\\
\deg f_{m,p,i-1,0}+\left(  q-1\right)  m & \text{if \thinspace}f_{m,p,i-1,q}=0
\end{cases}
&  &  \text{from Notation \ref{blocks}, Lemma \ref{relation}-1 and Corollary
\ref{fmpi0<>0}}\\
=  &  \
\begin{cases}
\deg\mathcal{R}_{m,r}f_{m,p,i-1,0}-\left(  m-r\right)  +qm & \text{if
\thinspace}f_{m,p,i-1,q}\neq0\\
\deg\mathcal{R}_{m,r}f_{m,p,i-1,0}+r+\left(  q-1\right)  m & \text{if
\thinspace}f_{m,p,i-1,q}=0
\end{cases}
&  &  \text{from Notation \ref{operation} and Lemma \ref{relation}-2}\\
=  &  \ \deg\mathcal{R}_{m,r}f_{m,p,i-1,0}+p-m &  &  \text{from }p=qm+r\\
=  &  \ \deg f_{m,p,i,0}+p-m &  &  \text{from the last claim}%
\end{align*}

\item $g_{p,i}^{\mathsf{b}}\leq\varphi\left(  m\right)  $

Note%
\begin{align*}
g_{p,i}^{\mathsf{b}}  &  =p+\operatorname*{tdeg}f_{m,p,i}-\deg f_{m,p,i-1} &
& \\
&  =p+\operatorname*{tdeg}f_{m,p,i,0}-\left(  \deg f_{m,p,i,0}+p-m\right)  &
&  \text{from the last claim}\\
&  =m+\operatorname*{tdeg}f_{m,p,i,0}-\deg f_{m,p,i,0} &  &  \text{by
simplifying}\\
&  \leq\varphi\left(  m\right)  &  &  \text{from Corollary \ref{degdiff}}%
\end{align*}

\end{enumerate}
\end{proof}

\noindent Finally we \textquotedblleft combine\textquotedblright\ the previous
results to prove the main result.

\begin{proof}
[Proof of Main Result (Theorem \ref{main})]Recall Lemma \ref{g_in_blockgaps}
\[
g\left(  \Phi_{mp}\right)  =\max\limits_{0\leq i\leq\varphi\left(  m\right)
-1}\left\{  g_{m,i}^{\mathsf{w}},\ \ g_{r,i}^{\mathsf{w}},\ \ g_{m,i}%
^{\mathsf{b}},\ \ g_{r,i}^{\mathsf{b}},\ \ g_{p,i}^{\mathsf{b}}\right\}
\]
By applying Lemmas \ref{gwm_i}, \ref{gwr_i}, \ref{gbm_i}, \ref{gbr_i} and
\ref{gbp_i} to the above expression, we finally have%
\[
g\left(  \Phi_{mp}\right)  =\varphi\left(  m\right)
\]

\end{proof}

\bibliographystyle{plain}
\bibliography{../../References/allrefs}
{}

\end{document}